\documentclass[12pt]{amsart}
\usepackage{etex}
\usepackage{amsfonts, amssymb, latexsym, epsfig, pdfsync}

\usepackage{amscd,amssymb}
\usepackage{verbatim}
\usepackage{pstricks}
\usepackage{pst-node}
\usepackage[mathscr]{euscript}
 \usepackage{tikz}
\usetikzlibrary{matrix,arrows}
\usepackage[normalem]{ulem}
\newcommand\redout{\bgroup\markoverwith
{\textcolor{red}{\rule[0.5ex]{2pt}{0.8pt}}}\ULon}

\input xy
\xyoption{all}

\newsymbol\pp 1275
\usepackage{tikz}
\usetikzlibrary{matrix,arrows,backgrounds,shapes.misc,shapes.geometric,patterns,calc,positioning}

\usepackage[colorlinks=true, pdfstartview=FitV, linkcolor=blue, citecolor=blue, urlcolor=blue]{hyperref}

\usepackage{fullpage}
\usepackage{graphicx}
\usepackage{enumerate}
\def\VR{\kern-\arraycolsep\strut\vrule &\kern-\arraycolsep}
\def\vr{\kern-\arraycolsep & \kern-\arraycolsep}

\newtheorem{theorem}{Theorem}

\newtheorem{lemma}[theorem]{Lemma}
\newtheorem{prop}[theorem]{Proposition}

\theoremstyle{definition}
\newtheorem{definition}[theorem]{Definition}

\newtheorem{rmk}[theorem]{Remark}

\newtheorem{obs}{Observation}

\newtheorem{ex}{Example}

\newcommand{\Hom}{\operatorname{Hom}}
\newcommand{\End}{\operatorname{End}}

\newcommand{\rep}{\operatorname{rep}}
\newcommand{\Proj}{\operatorname{Proj}}

\newcommand{\SI}{\operatorname{SI}}
\newcommand{\SL}{\operatorname{SL}}
\newcommand{\GL}{\operatorname{GL}}
\newcommand{\PGL}{\operatorname{PGL}}

\newcommand{\ZZ}{\mathbb Z}

\newcommand{\PP}{\mathbb P}

\newcommand{\Ker}{\operatorname{Ker}}

\newcommand{\Ima}{\operatorname{Im}}

\newcommand{\Id}{\operatorname{Id}}

\newcommand{\Mat}{\operatorname{Mat}}

\newcommand{\Stab}{\operatorname{Stab}}

\newcommand{\ddim}{\operatorname{\mathbf{dim}}}

\newcommand{\dd}{\operatorname{\mathbf{d}}}

\newcommand{\ff}{\operatorname{\mathbf{f}}}
\newcommand{\rr}{\mathbf{r}}

\newcommand{\nn}{\mathbf{n}}

\newcommand{\R}{\operatorname{\mathcal{R}}}
\newcommand{\M}{\operatorname{\mathcal{M}}}

\newcommand{\comp}{\operatorname{Comp}}

\newcommand{\rk}{\operatorname{rank}}

\newcommand\restr[2]{{
  \left.\kern-\nulldelimiterspace 
  #1 
  \vphantom{\big|} 
  \right|_{#2} 
  }}

\newcommand{\key}[1]{\emph{#1}}

\begin{document}
\title{Moduli spaces of representations of special biserial algebras}
\author{Andrew T. Carroll}
\address{DePaul University, Department of Mathematical Sciences, Chicago, IL, USA}
\email[Andrew T. Carroll]{acarro15@depaul.edu}

\author{Calin Chindris}
\address{University of Missouri-Columbia, Mathematics Department, Columbia, MO, USA}
\email[Calin Chindris]{chindrisc@missouri.edu}

\author{Ryan Kinser}
\address{University of Iowa, Department of Mathematics, Iowa City, IA, USA}
\email[Ryan Kinser]{ryan-kinser@uiowa.edu}

\author{Jerzy Weyman}
\address{University of Connecticut, Department of Mathematics, Storrs, CT, USA}
\email[Jerzy Weyman]{jerzy.weyman@uconn.edu} 
\thanks{C.C. was supported by NSA grant H98230-15-1-0022 and J.W. by NSF grant DMS-1400740}

\bibliographystyle{amsalpha}
\subjclass[2010]{16G20, 14D20}
\keywords{gentle algebras, moduli spaces, representations, regular irreducible components, special biserial algebras, varieties of circular complexes}

\begin{abstract}
We show that the irreducible components of any moduli space of
semistable representations of a special biserial algebra are always
isomorphic to products of projective spaces of various dimensions.
This is done by showing that irreducible components of varieties of
representations of special biserial algebras are isomorphic to irreducible components of products of varieties of circular complexes, and therefore normal, allowing us to apply recent results of the second and third authors on moduli spaces.
\end{abstract}

\maketitle
\setcounter{tocdepth}{1}
\tableofcontents

\section{Introduction}
Throughout, $K$ denotes an algebraically closed field of
characteristic zero. Unless otherwise specified, all quivers are
assumed to be finite and connected, and all algebras are assumed to be bound quiver algebras.

In this paper, we study representations of algebras within the general framework of Geometric Invariant Theory (GIT). This interaction between representations of algebras and GIT leads to the construction of moduli spaces of representations as solutions to the classification problem of semistable representations, up to S-equivalence. We point out that these moduli spaces can be arbitrarily complicated; indeed, arbitrary projective varieties can arise as moduli spaces of representations of algebras \cite{Hil,HZ2}.

Our goal in this paper is to understand these moduli spaces for special biserial algebras. The results we obtain here are, in fact, part of a program aimed at finding geometric characterizations of the representation type of bound quiver algebras. This line of research has attracted a lot of attention, see for example \cite{BleChi2009, BleChiHui-Zim15, Bob1, BS1, Bob5, Bob4, Carroll1, CC6, CC9, CC12, CC15, ChiKli16, Domo2, GeiSch, Rie, Rie-Zwa-1, Rie-Zwa-2, SW1}.

Special biserial algebras play a prominent role in the representation theory of algebras and related areas. Their indecomposable representations can be nicely described, however the number of $1$-parameter families needed to parametrize the $n$-dimensional indecomposables can grow faster than any polynomial in $n$. 
Algebraists have been interested in biserial and special biserial algebras for at least 50 years \cite{Tachikawa61,GP68,Fuller,SkW,MR801283,BR,Bleher98,EHIS,Herschend10}.
Special biserial algebras also naturally appear when studying tame blocks of group algebras of finite groups \cite{Janusz69,Rin75,DF78,Erd-1990,Roggen}.
Furthermore, gentle algebras and Brauer graph algebras, which are particular cases of special biserial algebras, have recently played an important role in the study of Jacobian and cluster algebras, see for example \cite{Labardini09, AsBruChaPla2010, GeiLab-FraSch2016, MarSch2014}.

Our main result is the following theorem which describes the irreducible components of moduli spaces for special biserial algebras.

\begin{theorem}\label{main-thm} Let $A$ be a special biserial algebra.  Then any irreducible component of a moduli space $\M(A,\dd)^{ss}_\theta$ is isomorphic to a product of projective spaces.
\end{theorem}

The isomorphism of the theorem results from a general decomposition theorem for moduli spaces proved in \cite{CKdecomp}.  The key geometric condition needed to apply this theorem is that certain representation varieties are normal. In this paper, we show in Proposition \ref{normality-prop-complete-gentle} that this condition holds in all cases relevant to special biserial algebras by reducing the consideration to varieties of circular complexes (see Sections \ref{normal-circular-complexes-sec} and \ref{normal-sec}).

\subsection*{Acknowledgements}  
This project began during a visit of the first three authors to the University of Connecticut. The authors would like to acknowledge the generous support of the Stuart and Joan Sidney Professorship of Mathematics endowment for making the visit possible.  We also thank Amelie Schreiber for participating in discussions about the project, Corrado de Concini for inspiring conversations, and the referees for comments improving the paper and simplifying some arguments.

\section{Representation varieties and moduli spaces} 
\subsection{Representation varieties} 
Since $K$ is algebraically closed, any finite-dimensional unital,
associative $K$-algebra $A$ can be viewed as a bound quiver algebra,
up to Morita equivalence; that is there exists a finite quiver $Q$, uniquely determined by $A$, and an admissible ideal $I$ of $KQ$ such that $A \simeq KQ/I$. Throughout, we will adopt the language of
representations of bound quivers. In particular, by abuse of
terminology, we refer to a representation of $Q$ satisfying the
relations in $I$ as a representation of $A$.  Whenever we work with a
set of generators $\R$ for $I$, we will always assume each generator
is a linear combination of paths with the same source and target
vertex.  If $\R$ generates an admissible ideal in $KQ$, we call the
pair $(Q,\R)$ a bound quiver. We assume throughout that $Q$ has has
finitely many vertices and finitely many arrows, and hence the algebra $KQ/\langle \R\rangle$ is finite-dimensional.

We write $Q_0$ for the set of vertices of $Q$, and $Q_1$ for its set
of arrows. For a dimension vector $\dd \in \ZZ^{Q_0}_{\geq 0}$, the
affine \emph{representation variety} $\rep(A,\dd)$ parametrizes the
$\dd$-dimensional representations of $(Q, \R)$ along with a fixed basis. Writing $ta$ and $ha$ for the tail and head of an arrow $a \in Q_1$, we have:
$$
\rep(A,\dd):=\{M \in \prod_{a \in Q_1} \Mat_{\dd(ha)\times \dd(ta)}(K) \mid M(r)=0, \text{~for all~} r \in
\R \}.
$$
Under the action of the change of base group $\GL(\dd):=\prod_{x\in
  Q_0}\GL(\dd(x),K)$, the orbits in $\rep(A,\dd)$ are in one-to-one
correspondence with the isomorphism classes of $\dd$-dimensional
representations of $(Q, \R)$. For more background on representation varieties, see \cite{Bon5,Zwarasurvey}.

In general, $\rep(A, \dd)$ does not have to be irreducible. Let $C$ be an irreducible component of $\rep(A, \dd)$. We say that $C$ is \emph{indecomposable} if $C$ has a non-empty open subset of indecomposable representations. We say that $C$ is a \emph{Schur component} if $C$ contains a Schur representation, in which case $C$ has a non-empty open subset of Schur representations; in particular, any Schur component is indecomposable. 

For dimension vectors $\dd_i \in \ZZ^{Q_0}_{\geq 0}, 1 \leq i \leq l$, and $\GL(\dd_i)$-invariant constructible subsets $C_i\subseteq \rep(A,\dd_i)$, $1 \leq i \leq l$, we denote by $C_1\oplus \ldots \oplus C_l$ the constructible subset of $\rep(A,\sum_{i=1}^l \dd_i)$ defined by  
$$C_1\oplus \ldots \oplus C_l=\{M \in \rep(A,\sum_{i=1}^l \dd_i) \mid M\simeq \bigoplus_{i=1}^l M_i\text{~with~} M_i \in C_i, \forall 1 \leq i \leq l\}.$$
As shown by de la Pe\~na in \cite{delaP} and Crawley-Boevey and Schr{\"o}er in \cite[Theorem 1.1]{C-BS} any irreducible component $C \subseteq \rep(A,\dd)$ satisfies a Krull-Schmidt type decomposition
$$
C=\overline{C_1 \oplus \ldots \oplus C_l}
$$
for some indecomposable irreducible components $C_i \subseteq \rep(A,\dd_i)$ with $\sum \dd_i=\dd$. Moreover, $C_1, \ldots, C_l$ are uniquely determined by this property. 

\subsection{Semi-Invariants}
Let $A=KQ/I$ be an algebra and $\dd \in \ZZ^{Q_0}_{\geq 0}$ a dimension vector of $A$. 
We are interested in the action of $\SL(\dd):=\prod_{x \in Q_0}\SL(\dd(x),K)$ on the representation variety $\rep(A,\dd)$. The resulting \key{ring of semi-invariants} $\SI(A,\dd):=K[\rep(A,\dd)]^{\SL(\dd)}$ has a weight space decomposition over the group $X^{\star}(\GL(\dd))$ of rational characters of $\GL(\dd)$:
$$\SI(A,\dd)=\bigoplus_{\chi \in X^\star(\GL(\dd))}\SI(A,\dd)_{\chi}.
$$
For each character $\chi \in X^{\star}(\GL(\dd))$, $$\SI(A,\dd)_{\chi}=\lbrace f \in K[\rep(A,\dd)] \mid g \cdot f= \chi(g)f \text{~for all~}g \in \GL(\dd)\rbrace$$ is called the \key{space of semi-invariants} on $\rep(A,\dd)$ of \key{weight} $\chi$. 

For a $\GL(\dd)$-invariant closed subvariety $C \subseteq \rep(A,\dd)$, we similarly define the ring of semi-invariants $\SI(C):=K[C]^{\SL(\dd)}$, and the space $\SI(C)_{\chi}$ of semi-invariants of weight $\chi \in X^{\star}(\GL(\dd))$.

Note that any $\theta \in \ZZ^{Q_0}$ defines a rational character $\chi_{\theta}:\GL(\dd) \to K^*$ by 
\begin{equation}
\chi_{\theta}((g(x))_{x \in Q_0})=\prod_{x \in Q_0}\det g(x)^{\theta(x)}.
\end{equation}
In this way, we get a natural epimorphism $\ZZ^{Q_0} \to X^{\star}(\GL(\dd))$; we refer to the rational characters of $\GL(\dd)$ as integral weights of $Q$ (or $A$). In case $\dd$ is a sincere dimension vector, this epimorphism is an isomorphism which allows us to identify $\ZZ^{Q_0}$ with $X^{\star}(\GL(\dd))$. 

\subsection{Moduli spaces of representations}\label{moduli-sec}
Let $(Q,\R)$ be a bound quiver, and $\theta \in \ZZ^{Q_0}$ an integral
weight of $Q$. Following King \cite{K}, a representation $M$ of $(Q,
\R)$ is said
to be \emph{$\theta$-semistable} if $\theta(\ddim M)=0$ and
$\theta(\ddim M')\leq 0$ for all subrepresentations $M' \leq M$. We
say that $M$ is \emph{$\theta$-stable} if $M$ is non-zero,
$\theta(\ddim M)=0$, and $\theta(\ddim M')<0$ for all
subrepresentations $0 \neq M' < M$. Finally, we call $M$ a
\key{$\theta$-polystable} representation if $M$ is a direct sum of $\theta$-stable representations.

Now, let $\dd$ be a dimension vector of $(Q,\R)$ and consider the (possibly empty) open subsets
$$\rep(A,\dd)^{ss}_{\theta}=\{M \in \rep(A,\dd)\mid M \text{~is~}
\text{$\theta$-semistable}\}$$
and $$\rep(A,\dd)^s_{\theta}=\{M \in \rep(A,\dd)\mid M \text{~is~}
\text{$\theta$-stable}\}$$
of $\dd$-dimensional $\theta$-(semi)stable representations of $(Q, \R)$. Using methods from Geometric Invariant Theory, King shows in \cite{K} that the projective variety
$$
\M(A,\dd)^{ss}_{\theta}:=\Proj\left(\bigoplus_{n \geq 0}\SI(A,\dd)_{n\theta}\right)
$$
is a GIT-quotient of $\rep(A,\dd)^{ss}_{\theta}$ by the action of $\PGL(\dd)$ where $\PGL(\dd)=\GL(\dd)/T_1$ and $T_1=\{(\lambda \Id_{\dd(x)})_{x \in Q_0} \mid \lambda \in k^*\} \leq \GL(\dd)$. Moreover, there is a (possibly empty) open subset $\M(A,\dd)^s_{\theta}$ of $\M(A,\dd)^{ss}_{\theta}$ which is a geometric quotient of $\rep(A,\dd)^s_{\theta}$ by $\PGL(\dd)$. We say that $\dd$ is a \emph{$\theta$-(semi)stable dimension vector} of $A$ if $\rep(A,\dd)^{(s)s}_{\theta} \neq \emptyset$. 

For a given $\GL(\dd)$-invariant closed subvariety $C$ of $\rep(A,\dd)$, we similarly define $C^{ss}_{\theta}, C^s_{\theta}$, $\M(C)^{ss}_{\theta}$, and $\M(C)^s_{\theta}$. We say that $C$ is a \emph{$\theta$-(semi)stable subvariety} if $C^{(s)s} \neq \emptyset$. 

From now on, let us assume that the character $\chi_{\theta} \in X^\star(\GL(\dd))$ induced by $\theta$ is not trivial, i.e. the restriction of $\theta$ to the support of $\dd$ is not zero, and denote by $G_{\theta}$ the kernel of $\chi_{\theta}$. Let $C$ be a $\theta$-semistable $\GL(\dd)$-invariant, irreducible, closed subvariety of $\rep(A,\dd)$. Then we claim that
$$
K[C]^{G_{\theta}}=\bigoplus_{n \geq 0} \SI(C)_{n \theta}.
$$
To justify this claim, consider first the action of the
$1$-dimensional torus $\GL(\dd)/G_{\theta}$ on $K[C]^{G_{\theta}}$.
Every character of this torus is induced from a character of $\GL(\dd)$ of the form $\chi_{n\theta}, n \in \ZZ$ since we work in characteristic 0.
(This can fail to be true in characteristic $p>0$ since $\det$ is a well-defined character of $\GL(d,K)/\ker (\det^p)=\GL(d,K)/\ker (\det)$.)
This yields the weight space decomposition $\bigoplus_{n \in \ZZ}
\SI(C)_{n \theta}$. It remains to show that $\SI(C)_{n \theta} =\{0\}$
for all integers $n<0$. For this we will use that $C^{ss}_{\theta}
\neq \emptyset$ and that $\chi_{\theta}$ is not the trivial
character. So, assume for a contradiction that there exists an integer
$n<0$ such that $\SI(C)_{n \theta} \neq \{0\}$. Then, we get a
representation $M \in C$ that is semistable with respect to both
$-\theta$ and $\theta$. In particular, this gives $\theta(\ddim M')=0$
for all subrepresentations $M' \leq M$. It is now clear that if $S_i$
is a simple representation that occurs as a composition factor in a
Jordan-H{\"o}lder filtration of $M$ then $\theta(i)=\theta(\ddim
S_i)=0$. Since $A$ is finite-dimensional, $M$ admits a
Jordan-H{\"o}lder filtration by the simples $S_i$. Thus, the restriction
of $\theta$ to the support of $\dd$ is zero (contradiction).

The restriction homomorphism $K[\rep(A,\dd)] \to K[C]$ remains surjective after taking $G_{\theta}$-invariants since $G_{\theta}$ is linearly reductive in characteristic zero. This surjective homomorphism $K[\rep(A,\dd)]^{G_{\theta}} \to K[C]^{G_{\theta}}$ of graded algebras gives rise to a closed embedding $\M(C)^{ss}_{\theta} \hookrightarrow \M(A,\dd)^{ss}_{\theta}$. In fact, the image of this embedding is precisely $\pi(C^{ss}_{\theta})$, where $\pi:\rep(A,\dd)^{ss}_{\theta} \to \M(A,\dd)^{ss}_{\theta}$ is the quotient morphism.

The points of $\M(C)^{ss}_{\theta}$ correspond bijectively to the (isomorphism classes of) $\theta$-polystable representations in $C$. Indeed, each fiber of $\pi:C^{ss}_{\theta} \to \M(C)^{ss}_{\theta}$ contains a unique closed $\GL(\dd)$-orbit in $C^{ss}_{\theta}$. On the other hand, as proved by King in \cite[Proposition 3.2(i)]{K}, these orbits are precisely the isomorphism classes of $\theta$-polystable representation in $C$. In fact, for any $M \in C^{ss}_{\theta}$, there exists a $1$-parameter subgroup $\lambda \in X_{\star}(G_{\theta})$ such that $\widetilde{M}:=\lim_{t \to 0} \lambda(t)M$ exists and is the unique, up to isomorphism, polystable representation in $\overline{\GL(\dd)M} \cap C^{ss}_{\theta}$.

The goal now is to explain how to decompose a given irreducible component of a moduli space of representations into smaller spaces which are easier to handle.  The following definition is from \cite{CKdecomp}.

\begin{definition}\label{def:thetastable}
Let $C$ be a $\GL(\dd)$-invariant, irreducible, closed subvariety of $\rep(A,\dd)$, and assume $C$ is $\theta$-semistable. Consider a collection $(C_i \subseteq \rep(A,\dd_i))_i$ of $\theta$-stable irreducible components such that $C_i \neq C_j$ for $i \neq j$, along with a collection of multiplicities $(m_i \in \ZZ_{>0})_i$.
We say that $(C_i, m_i)_i$ is a \key{$\theta$-stable decomposition of $C$} if, for a general representation $M \in C^{ss}_{\theta}$, its corresponding $\theta$-polystable representation $\widetilde{M}$ is in $C_1^{\oplus m_1} \oplus \cdots \oplus C_l^{\oplus m_l} $, and write
\begin{equation}\label{eq:thetastabledef}
C=m_1C_1\pp \ldots \pp m_l C_l.
\end{equation}
\end{definition}

Any $\GL(\dd)$-invariant, irreducible, closed subvariety $C$ of $\rep(A,\dd)$ with $C^{ss}_{\theta} \neq \emptyset$ admits a $\theta$-stable decomposition \cite[Proposition 3]{CKdecomp}. This decomposition controls the geometry of irreducible components of moduli spaces in the following sense.
Below, recall that the \key{$m^{th}$ symmetric power} $S^m(X)$ of a variety $X$ is the quotient of $\prod_{i=1}^m X$ by the action of the symmetric group on $m$ elements which permutes the coordinates.

\begin{theorem}\label{decomp-thm} \cite[Theorem 1]{CKdecomp}
Let $A$ be a finite-dimensional algebra and let $C \subseteq \rep(A,\dd)$ be a $\GL(\dd)$-invariant, irreducible, closed subvariety.  Let $C=m_1C_1\pp \ldots \pp m_l C_l$ be a $\theta$-stable decomposition of $C$ where $C_i \subseteq \rep(A,\dd_i)$, $1 \leq i \leq l$, are pairwise distinct $\theta$-stable irreducible components, and define $\widetilde{C} = \overline{C_1^{\oplus m_1} \oplus \cdots \oplus C_l^{\oplus m_l}}$.
\begin{enumerate}[(a)]
\item 
If $\M(C)^{ss}_{\theta}$ is an irreducible component of $\M(A,\dd)^{ss}_{\theta}$, then
$$\M(\widetilde{C})^{ss}_{\theta}=\M(C)^{ss}_{\theta}.$$
\item If $C_1$ is an orbit closure, then $$\M(\overline{C_1^{\oplus m_1} \oplus \cdots \oplus C_l^{\oplus m_l}})^{ss}_{\theta} \simeq \M(\overline{C_2^{\oplus m_2} \oplus \cdots \oplus C_l^{\oplus m_l}})^{ss}_{\theta}.$$ 
\item Assume now that none of the $C_i$ are orbit closures.
Then there is a natural morphism
\[
\Psi\colon  S^{m_1}(\mathcal{M}(C_1)^{ss}_{\theta}) \times \ldots \times S^{m_l}(\mathcal{M}(C_l)^{ss}_{\theta})  \to \M(\widetilde{C})^{ss}_{\theta}
\]
which is finite and birational. In particular, if $\M(\widetilde{C})^{ss}_{\theta}$ is normal then $\Psi$ is an isomorphism. 
\end{enumerate}
\end{theorem}

Note that given any (non-empty) moduli space $\M(A,\dd)^{ss}_{\theta}$, its irreducible components are all of the form $\M(C)^{ss}_{\theta}$ with $C$ a $\theta$-semistable irreducible component of $\rep(A,\dd)$. Thus, the theorem covers all the irreducible components of $\M(A,\dd)^{ss}_{\theta}$ and not just those of some special form.

Recall that a Schur-tame algebra is an algebra such that, in each dimension vector, all Schur representations (except possibly finitely many) come in a finite number of $1$-parameter families (see \cite[Definition 3]{CC15} for more details). For a Schur-tame algebra, each $\mathcal{M}(C_i)^{ss}_{\theta}$ appearing in the theorem has dimension 0 if $C_i$ is an orbit closure, and dimension 1 otherwise (see \cite[Proposition 12]{CC15}). Therefore, the dimension of $\M(C)^{ss}_{\theta}$ is precisely the sum of the multiplicities of the components which are not orbit closures.

\section{Varieties of circular complexes} \label{normal-circular-complexes-sec}
\subsection{Definition}
Fix a positive integer $l$ and an $l$-tuple of positive integers
$\nn=(n_i)_{i \in \ZZ/ l\ZZ}$ (for convenience in indices, we denote
the residue class of an integer $i$ modulo $l\ZZ$ by the same letter $i$). We are interested in the variety
$$
\comp(\nn):=\{(A_i)_{i \in \ZZ/l \ZZ} \in \prod_{i \in \ZZ/l\ZZ} \Mat_{n_{i+1} \times n_{i}}(K)\mid A_{i+1} A_i=0,\forall i \in \ZZ/ l\ZZ\},
$$
called the \emph{variety of circular complexes associated to $\nn$}. By convention, if $l=1$, we get the variety of matrices $A$ of size $n_0\times n_0$ with $A^2=0$.

Our goal in this section is to describe certain subvarieties of $\comp(\nn)$ given by rank conditions. 
It is useful to view $\comp(\nn)$ as a
representation variety for the following bound quiver. Consider the oriented cycle $\mathcal C$ with vertex set $\ZZ/l\ZZ$:
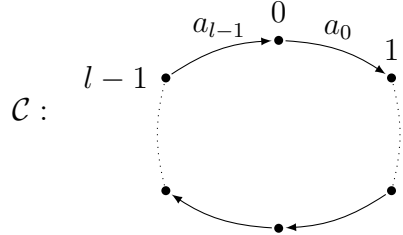
$$\mathcal{C}:~
\vcenter{\hbox{  
\begin{tikzpicture}[point/.style={shape=circle, fill=black, scale=.3pt,outer sep=3pt},>=latex]
   \node[point,label={above:$0$}] (0) at (0,0) {};
   \node[point,label={above:$1$}] (1) at (1.5,-.5) {};
   \node[point,label={left:$l-1$}] (l-1) at (-1.5,-.5) {};
   \node[point] (2) at (0,-2.5) {};
   \node[point] (3) at (1.5,-2) {};
   \node[point] (4) at (-1.5,-2) {};
   
   \draw[dotted] (1.5,-.5) to [bend left=15] (1.5,-2);
   \draw[dotted] (-1.5,-.5) to [bend right=15] (-1.5,-2);
   
   \path[->]
   (0) edge [bend left=15] node[midway, above] {$a_0$} (1)
   (3) edge [bend left=15]  (2)
   (2) edge [bend left=15]  (4)
   (l-1) edge [bend left=15] node[midway, above] {$a_{l-1}$} (0);
\end{tikzpicture} 
}}
$$
together with the admissible set of relations $\R:=\{a_{i+1}a_i \mid i
\in \ZZ/ l\ZZ \}$. Viewing $\nn$ as a
dimension vector of $\mathcal C$, $\comp(\nn)$ is precisely the
representation variety $\rep(K \mathcal{C}/\langle \R\rangle,\nn)$. Furthermore, $K \mathcal C
/\langle \mathcal R \rangle$ is a representation-finite algebra whose
indecomposable representations are:
\begin{enumerate}
\item the simples $S_i$, $i \in \ZZ/ l\ZZ$;

\item for each $i \in \ZZ/ l\ZZ$, the representation $E_{i,i+1}$ defined to be $K$ at vertices $i,i+1$, the identity map along the arrow $a_i$, and zero at all the other vertices and arrows.
\end{enumerate}
By convention, in case $l=1$, $\mathcal C$ is just the one-loop quiver with $\mathcal R=\{a^2\}$ where $a$ denotes the loop of $\mathcal C$. The indecomposable representations in this case are the simple $S_0$ at vertex $0$ of $\mathcal C$ and the $2$-dimensional representation $J_{2,0}$, given by the $2\times 2$ nilpotent Jordan block along the arrow $a$.  

Consequently, if $l>1$, any $\nn$-dimensional representation of
$(\mathcal C, \R)$, $M$ can be written as:
$$
M \simeq \bigoplus_{i \in \ZZ/ l\ZZ} E_{i,i+1}^{t_i} \oplus \bigoplus_{i \in \ZZ/ l\ZZ} S_{i}^{s_i},
$$
where the non-negative integers $t_i$ and $s_i$, $i \in \ZZ/l\ZZ$, satisfy the following conditions:
$$
t_{i-1}+t_i+s_i=n_i \text{~~and~~} t_i=\rk M(a_i), \forall i \in \ZZ/ l\ZZ.
$$ 
If $l=1$, these equations become $2t_0+s_0=n_0$ and $t_0=\rk M(a)$ where $M \simeq  J_{2,0}^{t_0} \oplus  S_0^{s_0}$. In either case,  we can see that $M$ is uniquely determined, up to isomorphism, by its dimension vector and the rank sequence $(\rk M(a_i))_{i \in \ZZ/ l\ZZ}$.

In what follows, by a \emph{rank sequence for $\nn$}, we mean a sequence $\rr=(r_i)_{i \in \ZZ/ l\ZZ}$ such that there exists an $M \in \rep(K\mathcal C / \langle \mathcal R \rangle, \nn)$ with $r_i=\rk M (a_i), \forall i \in \ZZ/ l \ZZ$; in particular, such an $\rr$ must satisfy $r_{i-1}+r_i\leq n_i$ for all $i \in \ZZ/ l\ZZ$.

\subsection{Subvarieties given by rank conditions}
For a rank sequence $\rr$ for $\nn$, consider the closed subvariety
$$
\comp(\nn,\rr):=\{(A_i)_{i \in \ZZ/ l\ZZ} \in \comp(\nn) \mid \rk A_i \leq  r_i, \forall i \in \ZZ/ l\ZZ \}.
$$
From the discussion above, we get that
$$
\comp^0(\nn,\rr):=\{(A_i)_{i \in \ZZ/ l\ZZ} \in \comp(\nn) \mid \rk A_i=  r_i, \forall i \in \ZZ/ l\ZZ \}
$$
is the $\GL(\nn)$-orbit in $\comp(\nn)$ of 
$$
M^0(\nn,\rr):=\bigoplus_{i \in \ZZ/ l\ZZ} E_{i,i+1}^{r_i} \oplus \bigoplus_{i \in \ZZ/ l\ZZ} S_{i}^{n_i-r_i-r_{i-1}}.
$$

\begin{lemma} \label{lemma-irr-comp-circular} For any rank sequence $\rr$ for $\nn$, the variety $\comp(\nn, \rr)$ is normal. 
\end{lemma}

\begin{proof} We first show that $$\comp(\nn,\rr)=\overline{\comp^0(\nn,\rr)}.$$
In what follows, we give an elementary proof of this equality. We point out that a more general approach can be found in Zwara's paper \cite{Zwa-1999}.

The containment $\supseteq$ is immediate from semi-continuity of rank; to show the opposite containment, we take an arbitrary point of $\comp(\nn,\rr)$ and produce an explicit degeneration from $\comp^0(\nn,\rr)$ to that point. Indeed, let $M \in \comp(\nn,\rr)$, and set $r_i':=\rk M(a_i)$ and $\epsilon_i:=r_i-r'_i$ for all $i \in \ZZ/l\ZZ$. Then $M$ belongs to the $\GL(\nn)$-orbit of 
$$
M^0(\nn,\rr'):=\bigoplus_{i \in \ZZ/ l\ZZ} E_{i,i+1}^{r'_i} \oplus \bigoplus_{i \in \ZZ/ l\ZZ} S_{i}^{n_i-r'_i-r'_{i-1}}.
$$
Next, for each $\lambda \in K$, consider the representation
$$
\bigoplus_{i \in \ZZ/ l\ZZ} E_{i,i+1}^{r'_i} \oplus \bigoplus_{i \in \ZZ/ l\ZZ} S_{i}^{n_i-r_i-r_{i-1}}\oplus \bigoplus_{i \in \ZZ/l\ZZ} E_{i,i+1}(\lambda)^{\epsilon_i},
$$
where $E_{i,i+1}(\lambda)$ is $K$ at vertices $i$ and $i+1$, $\lambda$ along the arrow $a_i$, and zero elsewhere. This representation is isomorphic to $M^0(\nn,\rr)$ for $\lambda \neq 0$, and to $M^0(\nn,\rr')$ when $\lambda=0$. So, we get that $M \in\overline{\GL(\nn)M^0(\nn,\rr)}=\overline{\comp^0(\nn,\rr)}$. This proves our claim that $\comp(\nn,\rr)=\overline{\comp^0(\nn,\rr)}$. In particular, $\comp(\nn,\rr)$ is an irreducible closed subvariety of $\comp(\nn)$ for any rank sequence $\rr$ for $\nn$. 

To see normality, simply note that orbit closures in varieties of circular complexes are examples of orbit closures of nilpotent representations of cyclicly-oriented type $\tilde{A}$ quivers, so \cite[Theorem~11.3]{Lusztig} gives that the $\comp(\nn, \rr)$ are locally isomorphic to an affine Schubert variety of type $A$. These varieties are known to be normal, for example by \cite[Theorem~8]{Faltingsloop}.
\end{proof}

\begin{rmk}\label{remark-comp} It is clear that $\comp(\nn)$ is covered by the $\comp(\nn,\rr)$ with $\rr$ rank sequences for $\nn$. The lemma above tells us that this is a cover by irreducible closed subvarieties. So, the irreducible components of $\comp(\nn)$ are among the $\comp(\nn,\rr)$'s. It is now immediate to see that the irreducible components of $\comp(\nn)$ are precisely the $\comp(\nn, \rr)$ with $\rr$ maximal (with respect to the coordinate wise order) rank sequences for $\nn$. 
\end{rmk}

We need a dimension bound for irreducible components of varieties of circular complexes. We give a geometric proof, but note that it can also be proven representation theoretically by giving a lower bound on the dimension of endomorphism rings of elements of $\comp(\nn,\rr)$, thought of as representations of the associated quiver with relations.

\begin{lemma}\label{lem:dimcomp}
The dimension of $\comp(\nn,\rr)$ is less than or equal to $\frac{1}{2}\sum_{i=0}^{l-1} n_i^2$.
\end{lemma}

\begin{proof} Consider the product of flag varieties
\[
Fl:=\prod_{i \in \ZZ/l\ZZ} Flag(r_{i-1}, n_i-r_i, V_i)
\]
where $Flag(r_{i-1}, n_i-r_i, V_i)$ denotes a two step flag variety (which becomes a Grassmannian if $r_{i-1}+r_i=n_i$) for each $i \in \ZZ/ l\ZZ$. 
Now consider the incidence variety:
\[
Z(\nn,\rr):= \lbrace (A_i , (R^1_i ,R^2_i))_{i \in \ZZ / l\ZZ}\in \comp(\nn,\rr)\times Fl \mid \Ima(A_{i-1})\subseteq R^1_i\subseteq R^2_i\subseteq \Ker (A_i), \forall i \in \ZZ / l\ZZ \rbrace .
\]
We have the two projections:
\vspace{-.5cm}
\begin{center}
\begin{tikzpicture}
    \node (Z) at (0,1) {$Z(\nn,\rr)$};
    \node (F) at (-2.5,0) {$Fl$};
    \node (C) at (2.5,0) {$\comp(\nn, \rr)$};
    \draw[->] (Z)--(F) node [midway,left] {$p$};
    \draw[->] (Z)--(C) node [midway,right] {$q$};
\end{tikzpicture}
\end{center}
The projection $p$ makes $Z(\nn ,\rr)$ a vector bundle over $\prod_{i
  \in \ZZ/l\ZZ} Flag(r_{i-1}, n_i-r_i, V_i)$, so $Z(\nn ,\rr)$ is
nonsingular, and the map $q$ is a birational isomorphism, since it is
an isomorphism over $\comp^0 (\nn ,\rr)$. In particular, \[\dim
\comp(\nn,\rr) = \dim Z(\nn,\rr) = \dim Fl + \dim
p^{-1}((R_i^1,R_i^2)_{i\in \ZZ/l\ZZ})\] where $(R_i^1, R_i^2)_{i\in
  \ZZ/l\ZZ}$ is an arbitrary flag in $Fl$. For such a fixed flag,
$p^{-1}((R_i^1,R_i^2)_{i\in \ZZ/l\ZZ}$ is isomorphic to
$\prod\limits_{i\in \ZZ/l \ZZ}\Hom_K(V_i/R_i^2, R_{i+1}^1)$, which has
dimension $\sum\limits_{i\in \ZZ/l\ZZ} r_i^2$. Meanwhile, the
formula for the dimension of a flag variety (see for example \cite[\S1.2]{brion}) in this case gives
\[\dim Fl=\sum\limits_{i\in \ZZ/l\ZZ}
  (r_{i-1}+r_i)(n_i-r_{i-1}-r_i)+r_{i-1}r_i.\] Therefore, 
\begin{align*} \dim \comp(\nn,\rr) &= \sum\limits_{i\in \ZZ/l\ZZ}
  (r_{i-1}+r_i)(n_i-r_{i-1}-r_i)+r_{i-1}r_i + r_i^2\\
&= \sum\limits_{i\in \ZZ/l\ZZ} (r_{i-1}+r_i)(n_i-r_{i-1}).\label{eq:dim-comp}\end{align*}

Let $k_i =n_i-r_i-r_{i-1}$. Note that with this notation, $\sum n_i^2 = \sum
(r_i+r_{i-1})^2 + 2k_i(r_i+r_{i-1})+k_i^2,$ while $\dim \comp(\nn,\rr)
= \sum (r_{i-1}+r_i)(r_i+k_i)$. Thus, we compute (suppressing the
index of summation where convenient):
\begin{align*}
  \sum n_i^2 - 2\dim \comp(\nn,\rr) &= \sum (r_i+r_{i-1})^2 +
                                      2k_i(r_i+r_{i-1})+k_i^2 -
                                      2(r_{i-1}+r_i)(r_i+k_i) \\
&= \sum (r_i+r_{i-1})(r_i+r_{i-1}+2k_i -2r_i-2k_i) + k_i^2\\
&= \sum (r_i+r_{i-1})(r_{i-1}-r_i) + \sum k_i^2\\
& =\sum r_{i-1}^2 - \sum r_i^2+\sum k_i^2.
\end{align*}
Note that the first two sums are equal, since indices are taken modulo
$l$, and the remaining sum is patently positive. Hence, the result
follows. 
\end{proof}

\section{Representation varieties of special biserial algebras and proof of the main result} \label{normal-sec}

Our main goal in this section is to check the normality condition in Theorem \ref{decomp-thm}(c), when the algebra in question is special biserial. We do this by reducing the considerations to varieties of circular complexes, whose irreducible components we already know are normal varieties (see Section \ref{normal-circular-complexes-sec}). 
\subsection{Special biserial and complete gentle algebras}
We begin by quickly recalling the definition of a special biserial
bound quiver algebra (see \cite{SkW}). A bound quiver $(Q,\mathcal R)$ is called a special biserial bound quiver if:

\begin{enumerate}
\item[(SB1)] for each vertex $v\in Q_0$ there are at most two arrows with head $v$, and at most two arrows with tail $v$;

\item[(SB2)] for every arrow $a \in Q_1$, there exists at most one arrow
  $b \in Q_1$ such that $ab \notin \langle \mathcal R\rangle$, and there exists at
  most one arrow $c \in Q_1$ such that $ca \notin \langle \mathcal R\rangle$.
\end{enumerate}

In what follows, by a quiver with relations $(Q, \R)$, we simply mean
a finite quiver $Q$ together with a finite set of (homogeneous) relations
where each relation is a linear combination of parallel paths of
length at least $2$. For a quiver with relations $(Q, \R)$, the
algebra $A=KQ/ \langle \R \rangle$ can be infinite-dimensional; it is
finite-dimensional precisely when $(Q,\R)$ is a bound quiver.

A quiver with relations $(Q,\R)$ is called \emph{gentle} if conditions (SB1) and (SB2) hold along with the following additional conditions:


\begin{enumerate} 
\item[(G1)] if $a_1$ and $a_2$ are two arrows with the same tail $v$ then, for any arrow $b$ with head $v$, precisely one of the $a_1b$ and $a_2b$ belongs to $\mathcal R$;

\item[(G2)] if $b_1$ and $b_2$ are two arrows with the same head $v$ then, for any arrow $a$ with tail $v$, precisely one of the $ab_1$ and $ab_2$ belongs to $\mathcal R$; 

\item[(G3)] $\mathcal R$ consists of paths of length two.
\end{enumerate}

A finite-dimensional algebra obtained from a gentle quiver with
relations by adding only monomial relations is known as a string algebra, and one obtained by adding arbitrary relations is a special biserial algebra. The finite-dimensional indecomposable representations for these algebras are well-known. Specifically, an indecomposable representation is either a projective, or string, or band representation (see \cite{BR}, \cite{Rin75}).

As explained by Ringel in \cite{Rin2011}, a special biserial algebra
can be viewed as a quotient of a rather special infinite-dimensional
gentle algebra, called a complete gentle algebra.

\begin{definition} Let $Q^*$ be a quiver and $\mathcal R^*$ a finite set of monomial relations of length two. We say that $(Q^*,\mathcal R^*)$ is a \emph{complete gentle quiver with relations} if for every vertex $x \in Q^*_0$, there are precisely two arrows ending at $x$ and precisely two arrows starting at $x$, and for every arrow $a \in Q^*_1$, there is precisely one arrow $a' \in Q^*_1$ and precisely one arrow $a'' \in Q^*_1$ such that $aa'$ and $a''a$ belong to $\mathcal R^*$. 

A \emph{complete gentle algebra} is an algebra isomorphic to $KQ^*/\langle \mathcal R^* \rangle$ with $(Q^*,\mathcal R^*)$ a complete gentle quiver with relations. Note that a complete gentle algebra is infinite-dimensional.
\end{definition}

For the convenience of the reader we include the following lemma due to Ringel (see \cite[Section 2]{Rin2011}).

\begin{lemma} Any special biserial algebra $A=KQ/I$ is a quotient of a complete gentle algebra, where the quiver has the same vertex set as $Q$.
\end{lemma}

\begin{proof} It is enough to show that an algebra given by a gentle
  quiver with relations $(Q,\R)$ is a quotient of a complete gentle
  algebra. Given a gentle quiver with relations $(Q,\R)$, we iteratively add arrows and relations to yield a complete gentle algebra. By conditions (SB1) and (SB2), it is clear that $\lvert Q_1\rvert \leq 2\lvert Q_0\rvert$ with equality precisely when $(Q, \R)$ is a complete gentle quiver with relations.

If $\lvert Q_1\rvert < 2 \lvert Q_0\rvert$ then there exist a vertex
$x$ that is the starting point of at most one arrow, $a$, and a vertex
$y$ that is the ending point of at most one arrow $b$. Define $Q^\prime$ to
be the quiver obtained from $Q$ by adding an arrow $c$ from $x$ to
$y$, and $\R^\prime$ to be the set of relations obtained by adding the
following length-two paths to $R$: $ca^\prime$ if $a^\prime$ is an arrow ending at
$x$ and $aa^\prime$ is not in $R$, and $b^\prime c$ if $b^\prime$ is an arrow starting at
$y$ and $b^\prime b$ is not in $R$. The pair $(Q^\prime,\R^\prime)$ is gentle and $\lvert
Q_1\rvert +1=\lvert Q_1^\prime \rvert \leq 2 \lvert Q^\prime_0\rvert$.  The addition of $2\lvert Q_0\rvert - \lvert Q_1\rvert$ arrows in this way
will produce a complete gentle algebra. The lemma now follows. 
\end{proof}

To describe the finite-dimensional indecomposable representations of complete gentle algebras, one uses the recipe developed for dealing with finite-dimensional gentle/string algebras. In particular, the finite-dimensional indecomposable representations are given again by bands and strings. This is due to the work of Ringel in \cite{Rin75}, and of Crawley-Boevey in \cite{CB6} where the more general case of finitely controlled or pointwise artinian indecomposable representations over infinite-dimensional string algebras is discussed.

\subsection{Representation varieties of complete gentle algebras}
Let $(Q^*,\mathcal R^*)$ be a complete gentle quiver with relations,  $\Lambda=KQ^*/ \langle \mathcal R^* \rangle$ its complete gentle algebra, and $\dd \in \ZZ_{\geq 0}^{Q^*_0}$ a dimension vector. 
In what follows, by an \emph{effective oriented cycle} of $(Q^*,\mathcal R^*)$, we mean an oriented cycle $\mathcal C=a_1 \ldots a_n$ of $Q^*$ such that $a_i \neq a_j$ for $i \neq j$, and $a_1a_2, \ldots, a_{n-1}a_n,\ a_na_1 \in \mathcal R^*$. (If $n=1$, we say that $\mathcal C=a_1$ is an effective oriented cycle if $a_1^2 \in \mathcal R^*$).

Since each arrow belongs to a unique effective oriented cycle, $Q_1$
can be written as a disjoint union of subsets of the form
$\{a_i\}_{i=1}^n$ ($n$ varying with the subset) where  $\mathcal C=a_1
\ldots a_n$ is an effective oriented cycle.  Therefore, the
representation variety $\rep(\Lambda,\dd)$ is a product of varieties
of circular complexes. Hence, the irreducible components of
$\rep(\Lambda,\dd)$ are normal varieties by Lemma
\ref{lemma-irr-comp-circular}. We remark that the same argument holds
in the case of arbitrary gentle algebras, whose representation
varieties have irreducible components that are products of circular
and ordinary complexes. However, we choose to work with complete
gentle algebras to avoid case-by-case analysis in the forthcoming
proofs. 

To describe the irreducible components in more concrete terms, let us recall that a sequence $\rr=(r_a)_{a \in Q^*_1}$ of non-negative integers is called a \emph{rank sequence for $\dd$} if there exists an $M \in \rep(\Lambda,\dd)$ with $r_a=\rk M(a), \forall a \in Q^*_1$. Note that this condition implies that $r_a+r_b \leq \dd(ta)$ for any two arrows $a,b$ with $ab \in \mathcal R^*$. A rank sequence for $\dd$ which is maximal with respect to the coordinate-wise order is called a \emph{maximal rank sequence for $\dd$}.

It follows from Lemma \ref{lemma-irr-comp-circular} that for any rank sequence $\rr$ for $\dd$, the set
$$
\rep(\Lambda,\dd,\rr):=\{M \in \rep(\Lambda,\dd) \mid \rk M(a) \leq r_a, \forall a \in Q^*_1\}
$$
is a normal subvariety of $\rep(\Lambda,\dd)$. Moreover, by Remark \ref{remark-comp}, the irreducible components of $\rep(\Lambda,\dd)$ are precisely those $\rep(\Lambda,\dd,\rr)$ with $\rr$ a maximal rank sequence for $\dd$. 

\begin{lemma}\label{lem-dim-mod}
Let $\Lambda$ be a complete gentle algebra and $\rr$ a rank sequence for a dimension vector $\dd$.  Then $\dim \rep(\Lambda,\dd,\rr) \leq \sum_{i \in Q^*_0} \dd(i)^2 = \dim \GL(\dd)$.
\end{lemma}
\begin{proof}
We have noted above that $\rep(\Lambda,\dd,\rr)$ is isomorphic to a product of varieties of complexes, say $\prod_j \comp(\nn^j, \rr^j)$.
Then we have by Lemma \ref{lem:dimcomp} that 
$$\dim \rep(\Lambda,\dd,\rr) = \sum_j \dim \comp(\nn^j, \rr^j) \leq \sum_j \frac{1}{2}\sum_i (n^j_i)^2.$$
Now for each vertex $i \in Q^*_0$, the value $\dd(i)$ appears exactly twice among the values $(n^j_i)$, since each vertex of $Q^*_0$ is a vertex for precisely two varieties of complexes (or the same one twice).  So the last double sum simplifies to $\sum_{i \in Q^*_0} \dd(i)^2$.
\end{proof}

\subsection{Proof of the main result}

For two given irreducible components $C\subseteq \rep(A,\ff)$ and $C' \subseteq \rep(A,\ff')$, we set: 
$$\hom_A(C,C')=\min \{\dim_K \Hom_A(X,Y) \mid (X,Y) \in C \times C'\}.$$
We are now ready to prove:

\begin{prop} \label{normality-prop-complete-gentle} Let $A=KQ/I$ be an arbitrary special biserial bound quiver algebra.
Let $C_i \subseteq \rep(A,\dd_i)$, $1 \leq i \leq m$, be irreducible components such that a general representation in each $C_i$ is Schur, and that $\hom_A(C_i,C_j)=0$ for all $1 \leq i, j \leq m$. Then 
$C:=\overline{C_1 \oplus \ldots \oplus C_m}$ is a normal variety. 
\end{prop}

\begin{proof} 
Take a complete gentle quiver with relations $(Q^*, \mathcal R^*)$ such that $Q^*_0 = Q_0$ and $A$ is a quotient of $\Lambda:=KQ^*/\langle \mathcal R^* \rangle$ by an ideal generated by arrows and admissible relations. We will find a rank sequence $\rr$ for $\dd:=\sum_i \dd_i$ such that $C=\rep(\Lambda,\dd,\rr)$, with the latter normal by Lemma \ref{lemma-irr-comp-circular}.

Now, for each $1 \leq i \leq m$, we have that $\dim C_i=\dim \GL(\dd_i)$ by the same arguments in \cite[Lemma 3]{CC13}, since $C_i$ is not an orbit closure.  But, we can also view $C_i \subseteq \rep(\Lambda,\dd_i)$,  where the maximal dimension of an irreducible component is $\dim \GL(\dd_i)$ by Lemma \ref{lem-dim-mod}. Therefore, $C_i$ has to be an irreducible component of $\rep(\Lambda,\dd_i)$ for each $1 \leq i \leq m$; in particular, each $C_i$ is a normal variety.

Next, for each $1 \leq i \leq m$, write $C_i=\rep(\Lambda,\dd_i,
\rr^i)$ where $\rr^i=(r^i_a)_{a \in Q^*_1}$ is a (maximal) rank
sequence for $\dd_i$. For the rank sequence $\rr:=\rr^1+\ldots
+\rr^m$, we have that $C=\overline{\bigoplus_{i=1}^m \rep(\Lambda,
  \dd_i, \rr^i)} \subseteq \rep(\Lambda, \dd, \rr)$, with the latter
being normal of dimension at most $\dim\GL(\dd)$ by Lemma \ref{lem-dim-mod}.  So we will show that $\dim C = \dim \GL(\dd)$ as well, forcing equality. 

Let $M \in C$ be a general element, so that $M$ is a direct sum of Schur representations with no nonzero morphisms between these summands.  Thus $\dim_K \End_A(M) = m = \dim \Stab_{\GL(\dd)}(M)$.  We also know by the general relation between dimensions of orbits and stabilizers that 
\begin{equation}\label{eq:dimGLd}
\dim \GL(\dd) = \dim \GL(\dd)\cdot M + \dim \Stab_{\GL(\dd)}(M) = \dim \GL(\dd)\cdot M + m.
\end{equation}
On the other hand, $C$ has a dense $m$-parameter family of distinct orbits, so for a general $M \in C$ we have that
\begin{equation}\label{eq:dimC}
\dim C = \dim \GL(\dd)\cdot M + m.
\end{equation}
Combining equations \eqref{eq:dimGLd} and \eqref{eq:dimC} then finishes the proof.
\end{proof}

\begin{proof}[Proof of Theorem \ref{main-thm}]  
Let $Y$ be an arbitrary irreducible component of $\M(A,\dd)^{ss}_{\theta}$. Then there exists an irreducible component $C$ of $\rep(A,\dd)$ such that $Y=\M(C)^{ss}_{\theta}$.
Consider the $\theta$-stable decomposition
$$C=m_1\cdot C_1 \pp \ldots \pp m_l\cdot C_l$$
as in Definition \ref{def:thetastable}.

By Theorem \ref{decomp-thm}, we can assume that $C=\overline{C_1^{m_1} \oplus \ldots \oplus C_l^{m_l}}$ and no $C_i$ is an orbit closure, so each $C_i$ must contain a dense family of band representations.
Furthermore, we have a morphism 
\[
\Psi\colon  S^{m_1}(\mathcal{M}(C_1)^{ss}_{\theta}) \times \ldots \times S^{m_l}(\mathcal{M}(C_l)^{ss}_{\theta})  \to \M(C)^{ss}_{\theta}
\]
which is surjective, finite, and birational.

Next, we claim that $\hom_{A}(C_i,C_j)=0$ for all $1 \leq i,j \leq l$. Indeed, for any $1 \leq i, j \leq l$, simply choose two non-isomorphic $\theta$-stable representations $X_i$ and $Y_j$ from $C_i$ and $C_j$, respectively; this is always possible since each $C_i$ is not an orbit closure. Then $\Hom_{A}(X_i,Y_j)=0$ and so $\hom_{A}(C_i,C_j)=0$. A general representation in each $C_i$ is Schur since the $C_i$ are $\theta$-stable. It now follows from Proposition \ref{normality-prop-complete-gentle} that $C$ (keeping in mind the reductions above) is normal.

Since $A$ is tame, we already know that each $\M(C_i)^{ss}_{\theta}$ is a rational projective curve (see, for example, \cite[Proposition 12]{CC15}). But $\M(C_i)^{ss}_{\theta}$ is also normal since $C_i$ is normal by the $m=1$ case in Proposition \ref{normality-prop-complete-gentle}; hence $\M(C_i)^{ss}_{\theta} \simeq \PP^1$ for all $1 \leq i \leq l$. We conclude that $\M(C)^{ss}_{\theta} \simeq \prod_{i=1}^l \PP^{m_i}$.
\end{proof}

\bibliography{biblio-special-biserial}\label{biblio-sec}
\end{document}